\documentclass[11pt,leqno]{article}
\usepackage[margin=1in]{geometry} 
\usepackage{amssymb,amsfonts,amsmath,bbm,mathrsfs,stmaryrd,mathtools}
\usepackage{xcolor}
\usepackage{url}

\usepackage{graphicx}

\usepackage{accents}

\usepackage{extarrows}

\usepackage[shortlabels]{enumitem}
\usepackage{tensor}

\usepackage{xr}
\usepackage[T1]{fontenc}
\usepackage[utf8]{inputenc}

\usepackage[colorlinks,
linkcolor=black!75!red,
citecolor=blue,
pdftitle={},
pdfproducer={pdfLaTeX},
pdfpagemode=None,
bookmarksopen=true,
bookmarksnumbered=true,
backref=page]{hyperref}

\usepackage{tikz}
\usetikzlibrary{arrows,calc,decorations.pathreplacing,decorations.markings,decorations.shapes,intersections,shapes.geometric,through,fit,shapes.symbols,positioning,decorations.pathmorphing}

\makeatletter
\newlength\zig@L
\newlength\zig@La
\newlength\zig@Lb

\newcommand{\xzigrightarrow}[2][]{%
  \mathrel{%
    \settowidth{\zig@La}{$\scriptstyle #2$}%
    \settowidth{\zig@Lb}{$\scriptstyle #1$}%
    \zig@L=\zig@La\relax
    \ifdim\zig@Lb>\zig@L \zig@L=\zig@Lb\fi
    \advance\zig@L by 2.2em\relax
    \tikz[baseline=-0.65ex]{%
      \draw[->,
            line cap=round,
            decorate,
            decoration={zigzag,segment length=4pt,amplitude=1.1pt}]%
        (0,0) -- (\zig@L,0)
        node[midway,above=2pt] {$\scriptstyle #2$}%
        \if\relax\detokenize{#1}\relax\else
          node[midway,below=2pt] {$\scriptstyle #1$}%
        \fi
      ;
    }%
  }%
}
\makeatother

\makeatletter
\newcommand{\squigjoin}{1mu} 

\def\sqleft@{\sim}                    
\def\sqmid@{\sim\mkern-\squigjoin}    

\def\rightsquigarrowfill@{%
  \arrowfill@{\sqleft@}{\sqmid@}{\mkern-4mu\succ}%
}

\newcommand{\xrightsquigarrow}[2][]{%
  \ext@arrow 0359\rightsquigarrowfill@{#1}{#2}%
}
\makeatother

\makeatletter
\newcommand*\circled[1]{\tikz[baseline=(char.base)]{
    \node[shape=circle, draw, inner sep=0pt, 
    minimum height={\f@size},] (char) {\vphantom{WAH1g}#1};}}
\makeatother

\makeatletter
\DeclareRobustCommand\widecheck[1]{{\mathpalette\@widecheck{#1}}}
\def\@widecheck#1#2{%
    \setbox\z@\hbox{\m@th$#1#2$}%
    \setbox\tw@\hbox{\m@th$#1%
       \widehat{%
          \vrule\@width\z@\@height\ht\z@
          \vrule\@height\z@\@width\wd\z@}$}%
    \dp\tw@-\ht\z@
    \@tempdima\ht\z@ \advance\@tempdima2\ht\tw@ \divide\@tempdima\thr@@
    \setbox\tw@\hbox{%
       \raise\@tempdima\hbox{\scalebox{1}[-1]{\lower\@tempdima\box
\tw@}}}%
    {\ooalign{\box\tw@ \cr \box\z@}}}
\makeatother

\usepackage{braket}

\usepackage[amsmath,thmmarks,hyperref]{ntheorem}
\usepackage{cleveref}

\newcommand\nthalias[1]{\AddToHook{env/#1/begin}{\crefalias{lemma}{#1}}}

\nthalias{definition}
\nthalias{example}
\nthalias{examples}
\nthalias{remark}
\nthalias{remarks}
\nthalias{convention}
\nthalias{notation}
\nthalias{construction}
\nthalias{sketch}
\nthalias{theoremN}
\nthalias{propositionN}
\nthalias{corollaryN}
\nthalias{lemma}
\nthalias{proposition}
\nthalias{corollary}
\nthalias{theorem}
\nthalias{conjecture}
\nthalias{question}
\nthalias{assumption}

\creflabelformat{enumi}{#2#1#3}

\crefname{section}{Section}{Sections}
\crefformat{section}{#2Section~#1#3} 
\Crefformat{section}{#2Section~#1#3} 

\crefname{subsection}{\S}{\S\S}
\AtBeginDocument{%
  \crefformat{subsection}{#2\S#1#3}%
  \Crefformat{subsection}{#2\S#1#3}%
}

\crefname{subsubsection}{\S}{\S\S}
\AtBeginDocument{%
  \crefformat{subsubsection}{#2\S#1#3}%
  \Crefformat{subsubsection}{#2\S#1#3}%
}

\theoremstyle{plain}

\newtheorem{lemma}{Lemma}[section]
\newtheorem{proposition}[lemma]{Proposition}

\newtheorem{theorem}[lemma]{Theorem}

\theoremstyle{plain}
\theoremnumbering{Alph}
\newtheorem{theoremN}{Theorem}

\theoremstyle{plain}
\theorembodyfont{\upshape}
\theoremsymbol{\ensuremath{\blacklozenge}}

\newtheorem{definition}[lemma]{Definition}

\newtheorem{examples}[lemma]{Examples}
\newtheorem{remark}[lemma]{Remark}

\newtheorem{notation}[lemma]{Notation}

\crefname{definition}{definition}{definitions}
\crefformat{definition}{#2definition~#1#3} 
\Crefformat{definition}{#2Definition~#1#3} 

\crefname{ex}{example}{examples}
\crefformat{example}{#2example~#1#3} 
\Crefformat{example}{#2Example~#1#3} 

\crefname{exs}{example}{examples}
\crefformat{examples}{#2example~#1#3} 
\Crefformat{examples}{#2Example~#1#3} 

\crefname{remark}{remark}{remarks}
\crefformat{remark}{#2remark~#1#3} 
\Crefformat{remark}{#2Remark~#1#3} 

\crefname{remarks}{remark}{remarks}
\crefformat{remarks}{#2remark~#1#3} 
\Crefformat{remarks}{#2Remark~#1#3} 

\crefname{convention}{convention}{conventions}
\crefformat{convention}{#2convention~#1#3} 
\Crefformat{convention}{#2Convention~#1#3} 

\crefname{notation}{notation}{notations}
\crefformat{notation}{#2notation~#1#3} 
\Crefformat{notation}{#2Notation~#1#3} 

\crefname{table}{table}{tables}
\crefformat{table}{#2table~#1#3} 
\Crefformat{table}{#2Table~#1#3}

\crefname{lemma}{lemma}{lemmas}
\crefformat{lemma}{#2lemma~#1#3} 
\Crefformat{lemma}{#2Lemma~#1#3} 

\crefname{proposition}{proposition}{propositions}
\crefformat{proposition}{#2proposition~#1#3} 
\Crefformat{proposition}{#2Proposition~#1#3} 

\crefname{propositionN}{proposition}{propositions}
\crefformat{propositionN}{#2proposition~#1#3} 
\Crefformat{propositionN}{#2Proposition~#1#3} 

\crefname{corollary}{corollary}{corollaries}
\crefformat{corollary}{#2corollary~#1#3} 
\Crefformat{corollary}{#2Corollary~#1#3} 

\crefname{corollaryN}{corollary}{corollaries}
\crefformat{corollaryN}{#2corollary~#1#3} 
\Crefformat{corollaryN}{#2Corollary~#1#3} 

\crefname{theorem}{theorem}{theorems}
\crefformat{theorem}{#2theorem~#1#3} 
\Crefformat{theorem}{#2Theorem~#1#3} 

\crefname{theoremN}{theorem}{theorems}
\crefformat{theoremN}{#2theorem~#1#3} 
\Crefformat{theoremN}{#2Theorem~#1#3} 

\crefname{enumi}{}{}
\crefformat{enumi}{#2#1#3}
\Crefformat{enumi}{#2#1#3}

\crefname{assumption}{assumption}{Assumptions}
\crefformat{assumption}{#2assumption~#1#3} 
\Crefformat{assumption}{#2Assumption~#1#3} 

\crefname{construction}{construction}{Constructions}
\crefformat{construction}{#2construction~#1#3} 
\Crefformat{construction}{#2Construction~#1#3} 

\crefname{sketch}{sketch}{Sketches}
\crefformat{sketch}{#2sketch~#1#3} 
\Crefformat{sketch}{#2Sketch~#1#3} 

\crefname{question}{question}{Questions}
\crefformat{question}{#2question~#1#3} 
\Crefformat{question}{#2Question~#1#3} 

\crefname{equation}{}{}
\crefformat{equation}{(#2#1#3)} 
\Crefformat{equation}{(#2#1#3)}

\numberwithin{equation}{section}

\theoremstyle{nonumberplain}
\theoremsymbol{\ensuremath{\blacksquare}}

\newtheorem{proof}{Proof}
\newcommand\pf[1]{\newtheorem{#1}{Proof of \Cref{#1}}}

\newcommand\bC{{\mathbb C}}

\newcommand\bF{{\mathbb F}}

\newcommand\bR{{\mathbb R}}
\newcommand\bS{{\mathbb S}}

\newcommand\bZ{{\mathbb Z}}

\newcommand\cC{{\mathcal C}}

\newcommand\cE{{\mathcal E}}

\newcommand\cX{{\mathcal X}}

\newcommand\ol{\overline}

\DeclareMathOperator{\Ad}{Ad}
\DeclareMathOperator{\id}{id}

\DeclareMathOperator{\End}{\mathrm{End}}

\DeclareMathOperator{\GL}{GL}

\newcommand{\comment}[1]{}

\title{Eigenvalue collision and exotic preservers on semisimple operators}
\author{Alexandru Chirvasitu}

\begin{document}

\date{}

\newcommand{\Addresses}{{
  \bigskip
  \footnotesize

  \textsc{Department of Mathematics, University at Buffalo}
  \par\nopagebreak
  \textsc{Buffalo, NY 14260-2900, USA}  
  \par\nopagebreak
  \textit{E-mail address}: \texttt{achirvas@buffalo.edu}

}}

\maketitle

\begin{abstract}
  We classify $n\times n$-matrix-valued continuous commutativity and spectrum preservers defined on spaces of (a) normal, (b) semisimple and (c) arbitrary $n\times n$ matrices with spectra contained in sufficiently connected subsets $\mathcal{X}\subseteq \mathbb{C}$, generalizing a number of results due to \v{S}emrl, Gogi\'{c}, Toma\v{s}evi\'c and the author among others. In case (a) these are always conjugations or transpose conjugations, while in cases (b) and (c) qualitatively distinct possibilities arise depending on the local regularity of the complex-conjugation map close to coincident-eigenvalue loci of $\mathcal{X}^n$.
\end{abstract}


\vspace{.5cm}



\section*{Introduction}

The present considerations fit under the broad umbrella of \emph{preserver-problem} literature, with \cite{2501.06840v2,MR4927632,MR4830482,Petek-TM,zbMATH01100760,MR1866032,zbMATH05302134} (to name but a few) as paradigmatic: the constraint that a map between various matrix spaces (e.g. a self-map of $M_n(\bC)$, or on unitary/special unitary/Hermitian/etc. $n\times n$ matrices) conserve various algebraic/analytic invariants rigidifies the situation sufficiently so as to afford a complete classification of the maps in question. The various threads in the now-rich tradition ultimately trace back to the \emph{Aupetit-Kaplansky problem} \cite[\S 1]{Aupetit} on whether linear spectrum-preserving surjections between complex semisimple Banach algebras must automatically be \emph{Jordan morphisms} in the sense of \cite[\S II.1.8.1]{mcc_jord_2004}. 

More specifically, the invariants preserved throughout most of the paper are operators' spectra and mutual commutativity, hence the pithy phrase \emph{CS preserver} (and variants) for maps which conserve these. The most direct entry point to the sequel is a curious phenomenon noted in the process of proving \cite[Theorem 0.1]{2501.06840v2} that the continuous $M_n(\bC)$-valued, $n\ge 3$ CS preservers on \emph{semisimple} (i.e. diagonalizable) $n\times n$ operators are exactly the conjugations $\Ad_T:=T(-)T^{-1}$ and transpose conjugations $\Ad_T\circ(-)^t$: while that statement is valid as made, there is a somewhat surprising (the above title's \emph{exotic}) other candidate: 
\begin{equation}\label{eq:snssns}
  \begin{gathered}
    M_{n,ss}(\bC)
    \ni
    \Ad_S N
    \xmapsto{\quad\Phi\quad}
    \Ad_{S^{-1}}N
    \in
    M_{n,ss}(\bC)\\
    S\ge 0,\ N^*N=NN^*\ \left(\text{i.e. $N$ is \emph{normal}}\right),
  \end{gathered}
\end{equation}
with the positivity $S\ge 0$ being the familiar \cite[Definition I.2.6.7]{blk} requirement that $\Braket{v\mid Sv}\ge 0$ for all $v\in \bC^n$ and the usual inner product $\Braket{-\mid -}$ on $\bC^n$. The only failure mode, it turns out, is discontinuity: although perhaps non-obviously, \Cref{eq:snssns} \emph{is} a well-defined CS preserver. To better isolate that precise point of failure, it will be profitable to separate two aspects of the problem:
\begin{itemize}[wide]
\item on the one hand the inherent geometry underlying \Cref{eq:snssns}, in the guise of its effect on the respective eigenspaces of its argument and image: for \emph{simple} diagonal operators (i.e. those with simple spectra),
  \begin{equation*}
    \forall i\left(\ker\left(\lambda_i-\Ad_SN\right)=\ell_i\right)
    \xRightarrow{\quad}
    \forall i\left(\ker\left(\lambda_i-\Ad_SN\right)=\ell'_i:=\left(\bigoplus_{j\ne i}\ell_j\right)^{\perp}\right);
  \end{equation*}

\item on the other, whatever residual phenomena result from eigenvalue dynamics only. 
\end{itemize}
As $\left(\ell_i\right)_i\mapsto \left(\ell'_i\right)_i$ is a perfectly well-defined continuous self-map on the space of linearly independent line $n$-tuples in $\bC^n$ (the \emph{eversion} map of \cite[(0-2)]{2601.11455v2}), it is the latter factor that must bring about the aforementioned discontinuity. This is visible in the proof of \cite[Proposition 2.14]{2501.06840v2}, which traces back the issue, via the semisimple-operator functional calculus discussed in \cite{zbMATH06285212}, to the failure of the complex conjugation map $\overline{(-)}$ to be sufficiently regular.

To elaborate, recall \cite[Definition 3.1]{zbMATH06285212}'s \emph{difference operator} mapping symmetric $n$-distinct-variable functions to again symmetric $(n+1)$-variable such:
\begin{equation*}
  \Delta f(z_0,\cdots,z_n)
  :=
  \frac{f(z_1,\cdots,z_n)-f(z_0,\cdots,z_{n-1})}{z_n-z_0}.
\end{equation*}
Said irregularity amounts to the iterated ($n$-variable) $\Delta^{n-1}\ol{(-)}$ not being locally bounded around diagonal elements $(z,\cdots,z)\in \bC^n$. The crucial hinge, then, in disqualifying \Cref{eq:snssns} is precisely the eigenvalue coincidence (or \emph{collision}) the paper's title hints at. 

The following statement compacts and paraphrases \Cref{th:nrm,th:ss,th:all} below, at the cost of some sharpness and strength.  

\begin{theoremN}\label{thn:3variants}
  Let $n\in \bZ_{\ge 3}$ and $\cX\subseteq \bC$ a perfect set with its space $\cC^n(\cX)$ of distinct $n$-tuples connected.
  \begin{enumerate}[(1),wide]
  \item\label{item:thn:3variants:nrm} The continuous $M_n(\bC)$-valued CS preservers on normal $n\times n$ matrices with spectra contained in $\cX$ are precisely the conjugations and transpose conjugations.  

  \item\label{item:thn:3variants:ss} The continuous $M_n(\bC)$-valued CS preservers on semisimple $n\times n$ matrices with spectra contained in $\cX$ are precisely the conjugations, transpose conjugations, and these composed with \Cref{eq:snssns} precisely when $\Delta^{n-1}\ol{(-)}$, as a function on $\cC^n(\cX)$, is locally bounded around every $(z,\cdots,z)\in \cX^n$. 

  \item\label{item:thn:3variants:all} The continuous $M_n(\bC)$-valued CS preservers on arbitrary $n\times n$ matrices with spectra contained in $\cX$ are precisely the conjugations, transpose conjugations, and these composed with \Cref{eq:snssns} precisely when $\Delta^{n-1}\ol{(-)}$, as a function on $\cC^n(\cX)$, has finite limits at all $(z,\cdots,z)\in \cX^n$. 
  \end{enumerate}
\end{theoremN}


\section{Sufficient connectedness and commutativity/spectrum preservation}\label{se:cs}

$\sigma(\bullet)\subset \bC$ stands for the spectrum of an operator in $M_n(\Bbbk)$, with subscript-decorated $\sigma^{-}(\cX)$ denoting various classes of operators with spectra contained in subsets $\cX\subseteq \bC$:
\begin{equation}\label{eq:sigma.bull}
  M_n(\Bbbk)
  \supseteq
  \sigma_{\bullet}^{-}(\cX)
  :=
  \left\{
    T\in M_n(\Bbbk)
    \ :\
    \left(\sigma(T)\subseteq \cX\right)
    \ 
    \wedge
    \ 
    T
    \left[
      \begin{aligned}
        \text{ arbitrary}&\text{ if $\bullet=\text{blank}$}\\
        \text{ semisimple}&\text{ if $\bullet=ss$}\\
        \text{ normal}&\text{ if $\bullet=\perp$}
      \end{aligned}
    \right.
  \right\}.
\end{equation}

The main results all concern spaces of operators in $M_n(\bC)$ with spectra contained in ``sufficiently connected'' subsets $\cX\subseteq \bC$. Some auxiliary language and notation will help express the relevant conditions.

\begin{notation}\label{not:config.grph}
  Consider a set $\cX$ and a positive integer $n\in \bZ_{>0}$.
  \begin{enumerate}[(1),wide]
  \item The \emph{coincidences} of $\mathbf{x}=(x_i)_{i=1}^n\in \cX^n$ are the pairs $\{i\ne j\}\subset \{1..n\}$ with $x_i=x_j$. The set of coincidences of $\mathbf{x}$ is denoted by $\cE^n(\mathbf{x})$. 

  \item The \emph{$n^{th}$ configuration space} \cite[Definition 1.1]{zbMATH05785888} of $\cX$ is
    \begin{equation*}
      \cC^n(\cX)
      :=
      \left\{\mathbf{x}=(x_1,\cdots,x_n)\in \cX^n\ :\ x_i\text{ distinct}\right\}
      =
      \left\{\mathbf{x}\ :\ \cE^n(\mathbf{x})=\emptyset\right\}.
    \end{equation*}

  \item Assume $\cX$ equipped with a topology, which then induces topologies on the configuration spaces $\cC^n(\cX)\subseteq \cX^n$.

    For a connected component $\cC\subseteq \cC^n(\cX)$ the (loop-free, undirected) \emph{$\cC^n$-graph $\Gamma^n(\cC)$} has $\{1..n\}$ as vertices and an edge $\{i\ne j\}$ whenever
    \begin{equation*}
      \ol{\cC}
      \cap
      \ol{\left(\cC\triangleleft(i,j)\right)}
      \cap
      \left(\cE^n\right)^{-1}\left(\{i,j\}\right)
      \ne\emptyset
      \quad
      \left(\text{closures in $\cX^n$}\right),
    \end{equation*}
    where ``$\triangleleft$'' denotes the right symmetric-group action
    \begin{equation*}
      \left(\mathbf{x}=(x_i)_{i=1}^n\in \cC^n(\cX)\right)\triangleleft\left(\theta\in S_n\right)
      :=
      \left(x_{\theta i}\right)_{i=1}^n
    \end{equation*}
    (in turn inducing an action on the space of connected components of $\cC^n(\cX)$).

    The terminology and notation apply to points $\mathbf{x}\in \cC^n(\cX)$: $\Gamma^n(\mathbf{x})$ is the $\cC^n$-graph of the connected component containing $\mathbf{x}$.
  \end{enumerate}
\end{notation}

\begin{examples}
  \begin{enumerate}[(1),wide]
  \item For $\cX:=\bS^1$ the $\cC^n$-graph of an arbitrary $\mathbf{x}\in \cC^n(\cX)$ is an $n$-cycle: the edges are precisely the pairs $\{i\ne j\}$ for $x_i$ and $x_j$ adjacent in a counterclockwise enumeration along the circle.

  \item Similarly, for $\cX:=\bR$ (or indeed, any connected non-singleton subset of $\bR^n$) $\cC^n$-graphs are $(n-1)$-edge paths: $\{i<j\}$ is an edge precisely when $x_i$ and $x_j$ are consecutive in an increasing enumeration of the components of $\mathbf{x}$.
  \end{enumerate}
\end{examples}

Recall \cite[Definition 3.1]{wil_top} that a topological space is \emph{perfect} if it clusters to all of its points: every $x\in \cX$ lies in the closure $\overline{\cX\setminus\{x\}}$. 

\begin{theorem}\label{th:nrm}
  Let $n\in \bZ_{\ge 3}$ and $\cX\subseteq \bC$ a perfect space with all graphs $\Gamma^n(\mathbf{x})$, $\mathbf{x}\in \cC^n(\cX)$ containing $n$-cycles. Assume furthermore that
  \begin{itemize}[wide]
  \item $S_n$ acts transitively on the connected components of $\cC^n(\cX)$;
  \item and that action is not free (i.e. the isotropy groups are not trivial). 
  \end{itemize}

  The continuous CS preservers $M_n(\bC)\supseteq \sigma_{\perp}^-(\cX)\xrightarrow{\phi} M_n(\bC)$ are precisely the conjugations or transpose conjugations by elements $\GL_n(\bC)$. 
\end{theorem}

\begin{remark}\label{re:wkr}
  Note the anticipated weakening in casting \Cref{th:nrm} as \Cref{thn:3variants}\Cref{item:thn:3variants:nrm}: for the circle $\cX:=\bS^1$, which does meet the constraints imposed in the former statement, $\cC^n(\cX)$ will be disconnected as soon as $n\ge 3$.
\end{remark}

As maps of the form $\Ad_T:=T(-)T^{-1}$ and $\Ad_T\circ(-)^t$ evidently meet the requirements, the discussion will focus exclusively on the converse (i.e. the claim that this covers all possibilities). The proof strategy is very much in line with that of \cite[Theorem 2.1]{2501.06840v2}, relying on the one hand on symmetric-group combinatorics (\cite[Proposition 2.2]{2501.06840v2}, \cite[Proposition 1.4]{2505.19393v3}) in reducing the problem to maximal abelian subalgebras of $\sigma^{-1}_{\perp}(\cX)$ and on the other hand proceeding thence via one variant \cite[Theorem 0.1]{2601.11455v2} of the \emph{fundamental theorem of projective geometry} (e.g. \cite[Theorem 3.1]{zbMATH01747827}).

The first ingredient, then, is as follows (with the subscripts ``$\bullet$'' as in \Cref{eq:sigma.bull}); note that the matrix size need not be restricted to $n\ge 3$ only. 

\begin{proposition}\label{pr:max.ab}
  Under the hypotheses of \Cref{th:nrm} with $n\in \bZ_{\ge 1}$, its conclusion holds for CS preservers $\sigma_{\bullet}^-(\cX)\xrightarrow{\phi}M_n(\bC)$ restricted to maximal abelian semisimple subalgebras of the domain.
\end{proposition}
\begin{proof}
  A maximal abelian semisimple subalgebra is a conjugate of that of diagonal operators in $\sigma_{\bullet}^-(\cX)$, so it suffices to work with that diagonal algebra $D\cong \cX^n$ (having identified the $n$-tuple $(\lambda_i)_i\subset \cX^n$ with the corresponding diagonal matrix). Observe also that the assumed perfection of $\cX$ ensures the density of $\cC^n(\cX)\subseteq \cX^n$, so that (given the continuity of $\phi$) it is enough to prove the statement for the subspace $D_s\subseteq D$ of simple diagonal operators.

  Given the transitivity of the $S_n$-action on the connected components of $\cC^n(\cX)$, we have to argue that $\phi$ restricts to a conjugation $\Ad_T$ on
  \begin{equation*}
    \cC^n(\cX)
    =
    \bigcup_{\theta\in S_n}\left(\cC\triangleleft\theta\right)
    \quad
    \left(\cC\subseteq \cC^n(\cX)\cong D_s\right)
  \end{equation*}
  for a fixed connected component $\cC\subseteq \cC^n(\cX)$. In conjunction with CS preservation, the connectedness of $\cC$ implies that for all permutations $\theta\in S_n$
  \begin{equation*}
    \left(\cC\triangleleft \theta\right)
    \ni
    \mathbf{x}
    \xmapsto{\quad\phi\quad}
    \mathbf{x}\triangleleft\theta'
    \in
    \left(\cC\triangleleft \theta\theta'\right)
  \end{equation*}
  for $\theta'\in S_n$ depending only on $\theta$ (rather than $\mathbf{x}\in \cC\triangleleft\theta$). By that selfsame CS preservation,
  \begin{equation*}
    \forall\left(\left\{i,j\right\}\in \Gamma^n(\cC)\right)
    \forall\left(\theta\in S_n\right)
    \bigg(
    \left(
      (i\ j)\cdot \theta
      =
      \theta\cdot \Ad_{\theta^{-1}}(i\ j)
    \right)'
    \in
    \left\{\theta',\ \Ad_{\theta^{-1}}(i\ j)\cdot \theta'\right\}
    \bigg).
  \end{equation*}
  This implies by \cite[Proposition 1.4]{2505.19393v3} that $\theta\mapsto \theta'$ is either constant or of the form
  \begin{equation*}
    \theta
    \xmapsto{\quad}
    \theta':=
    \left(\tau\cdot \theta\right)^{-1},
    \text{ fixed }\tau\in S_n,
  \end{equation*}
  the latter possibility being ruled out by the non-triviality of the isotropy group of $\cC$ in $S_n$. Having assumed $\phi|_{\cC}=\id$ (i.e. $\id'=\id$) by composing $\phi$ with a conjugation, $\theta'=\id$ for all $\theta$ and $\phi=\id$ on $\cC^n(\cX)$.
\end{proof}

In preparation for the proof of \Cref{th:nrm}, we follow \cite[Introduction]{2601.11455v2} in writing
\begin{equation*}
  \bF(V):=\left\{\text{spanning line $n$-tuples}\right\}
  ,\quad
  \bF^{\perp}(V):=\left\{\text{spanning orthogonal line $n$-tuples}\right\}
\end{equation*}
for $n$-dimensional vector (or Hilbert) spaces $V$ over $\Bbbk\in \left\{\bR,\bC\right\}$.

\pf{th:nrm}
\begin{th:nrm}
  Each orthogonal-line tuple $\mathbf{\ell}=\left(\ell_i\right)_{i=1}^n\in \bF(\bC^n)$ determines a unique maximal abelian subalgebra $D_{\mathbf{\ell}}\le \sigma^-_{\perp}(\cX)$ whose simple operators have the $\ell_i$ as their eigenspaces. \Cref{pr:max.ab} provides a continuous map
  \begin{equation*}
    \bF^{\perp}(\bC^n)
    \xrightarrow{\quad\Theta\quad}
    \bF(\bC^n)
  \end{equation*}
  determined uniquely by
  \begin{equation*}
    \begin{gathered}
      D_{\ell\in \bF^{\perp}(\bC^n)}
      \ni
      \left(\text{simple }T\right)
      \xmapsto{\quad\phi\quad}
      \phi T
      \in 
      D_{\Theta\ell\in \bF(\bC^n)}\\
      \ell_i=\ker\left(\lambda_i-T\right)
      \iff
      (\Theta\ell)_i=\ker\left(\lambda_i-\phi T\right),
    \end{gathered}    
  \end{equation*}
  with the following properties:
  \begin{itemize}[wide]
  \item $\Theta$ is $S_n$-equivariant for the right $S_n$-actions $(\mathbf{\ell}\triangleleft\sigma)_i=\ell_{\sigma i}$, $\mathbf{\ell}\in \bF^{\bullet}(\bC^n)$;

  \item and $\Theta$ preserves the \emph{$\pi$-linking} relations $\sim_{\pi}$ for partitions
    \begin{equation*}
      \pi=\left(\pi_j\right)_{j=1}^s
      ,\quad
      \{1..n\}=\bigsqcup_{j=1}^s \pi_j
    \end{equation*}
    in the sense of \cite[Definition 0.2]{2601.11455v2}:
    \begin{equation*}
      \mathbf{\ell}
      \sim_{\pi}
      \mathbf{\ell'}
      \iff
      \forall\left(1\le j\le s \right)
      \left(\bigoplus_{i\in \pi_j}\ell_i = \bigoplus_{i\in \pi_j}\ell'_i\right).
    \end{equation*}
  \end{itemize}
  \cite[Theorem 0.3]{2601.11455v2} then ensures that $\Theta$ is of the form $\left(\ell_i\right)_i\mapsto \left(T\ell_i\right)_i$ for invertible $T$ acting either linearly or conjugate-linearly on $\bC^n$, whence the conclusion: $\phi$ is $\Ad_T$ in the first case and a transpose conjugation in the second:
  \begin{equation*}
    \forall\left(S\in \sigma^{-}_{\perp}(\cX)\right)
    \left(\phi S = \Ad_T S^* = \Ad_{TJ} S^t\right),
  \end{equation*}
  where $\bC^n\xrightarrow[\cong]{J}\bC^n$ is the conjugate-linear operator acting as the identity on the standard basis effecting the identification $\End(\bC^n)\cong M_n(\bC)$ assumed throughout. 
\end{th:nrm}

There are $\sigma^{-}_{ss,\text{blank}}(\cX)$ analogues of \Cref{th:nrm}, by necessity more elaborate due to new possibilities for what CS preservers might look like depending on the topology of $\cX$. To prepare the ground for those statements, recall some notions of function regularity from \cite[\S 3, especially Definition 3.1]{zbMATH06285212}.

\begin{definition}\label{def:fn.reg}
  \begin{enumerate}[(1),wide]
  \item For a symmetric $n$-variable function $f$ defined on $\cC^n(\Gamma_0)$ for a subset $\Gamma_0\subseteq \Gamma$ of an abelian group $\left(\Gamma,+\right)$ set
    \begin{equation*}
      \Delta f(z_0,\cdots,z_n)
      :=
      \frac{f(z_1,\cdots,z_n)-f(z_0,\cdots,z_{n-1})}{z_n-z_0}
    \end{equation*}
    (symmetric, defined on $\cC^{n+1}(\Gamma_0)$). Note that $\Delta$ can be iterated: $\Delta^0=\id$, $\Delta^k=\Delta\circ\cdots\circ\Delta$ ($k$-fold). 
    
  \item A function $\bC\supseteq \cX\xrightarrow{f}\bC$ is \emph{of class $B^k_{\Delta}(\cX)$} (or just \emph{$B^k_{\Delta}$} when $\cX$ is understood) if $\Delta^k f$ is locally bounded around every diagonal element $(z,\cdots,z)\in \cX^{k+1}$ for cluster points $z\in \cX$.
    
  \item $f$ as in the preceding item is \emph{of class $C^k_{\Delta}(\cX)$} (or just \emph{$C^k_{\Delta}$}) when $\Delta^k f$ has finite limits at all diagonal $(z,\cdots,z)\in \cX^{k+1}$ for cluster points $z\in \cX$.
  \end{enumerate}
\end{definition}

Also helpful:

\begin{notation}\label{not:frm.eigs}
  For $\mathbf{\ell}=(\ell_i)_i\in \bF(V)$ and $\mathbf{\lambda}=(\lambda_i)\in \bC^n$ the symbols $\left(\mathbf{\ell}\mid\mathbf{\lambda}\right)=\left(\ell_i\mid\lambda_i\right)_i$ denote the semisimple operator on $V$ with eigenvalue $\lambda_i$ along $\ell_i$. 
\end{notation}

\begin{theorem}\label{th:ss}
  Let $n\in \bZ_{\ge 3}$ and $\cX\subseteq \bC$ a perfect space with all graphs $\Gamma^n(\mathbf{x})$, $\mathbf{x}\in \cC^n(\cX)$ containing $n$-cycles. Assume furthermore that
  \begin{itemize}[wide]
  \item $S_n$ acts transitively on the connected components of $\cC^n(\cX)$;
  \item and that action is not free (i.e. the isotropy groups are not trivial). 
  \end{itemize}
  The continuous CS preservers
  \begin{equation*}
    M_n(\bC)\supseteq\sigma^{-}_{ss}(\cX)\xrightarrow{\phi}M_n(\bC)
  \end{equation*}
  are as follows. 
  
  \begin{enumerate}[(1),wide]
  \item\label{item:th:ss:notb} If complex conjugation is not $B^{n-1}_{\Delta}(\cX)$, precisely the conjugations and transpose conjugations.

  \item\label{item:th:ss:b} If complex conjugation is $B^{n-1}_{\Delta}(\cX)$, precisely the conjugations, transpose conjugations and their compositions with
    \begin{equation}\label{eq:evers.mtr}
      \sigma^{-}_{ss}(\cX)
      \ni
      \left(\mathbf{\ell}\mid\mathbf{\lambda}\right)
      \xmapsto{\quad\Phi\quad}
      \left(\mathbf{\ell'}\mid\mathbf{\lambda}\right)
      \in
      M_{n,ss}(\bC)
      ,\quad
      \forall\left(1\le i\le n\right)
      \left(\ell'_i:=\left(\bigoplus_{j\ne i} \ell_j\right)^{\perp}\right).
    \end{equation}
  \end{enumerate}
\end{theorem}
\begin{proof}
  First assume $\phi$ given. \Cref{pr:max.ab} applies just as well in the present context as it did in that of \Cref{th:nrm} whence, as in the latter's proof, a symmetric, partition-linking-preserving self-map $\Theta$ on $\bF(\bC^n)$ mapping the eigenspaces of $T\in \sigma^{-}_{ss}(\cX)$ to those of $\phi T$. This is one point of departure from the earlier proof: \cite[Theorem 0.4]{2601.11455v2} now classifies the possible $\Theta$ as 
  \begin{equation*}
    \Theta\in \left\{\Theta_T,\ \Theta_T\circ\Theta_{ev}\ :\ \bC^n\xrightarrow[\text{(conjugate-)linear $\cong$}]{T}\bC^n\right\}
  \end{equation*}
  for
  \begin{equation*}
    \left(\ell_i\right)\xrightarrow{\quad\Theta_T\quad}\left(T\ell_i\right)
    \quad\text{and}\quad
    \left(\ell_i\right)_i
    \xrightarrow[\quad]{\quad\Theta_{ev}\quad}
    \left(\left(\bigoplus_{j\ne i}\ell_j\right)^{\perp}\right)_i.
  \end{equation*}
  This already reduces the possibilities for $\phi$ to those listed in \Cref{item:th:ss:b}, and the proof will be complete once we argue that \Cref{eq:evers.mtr} is continuous \emph{precisely} when $\overline{(\bullet)}\in B^{n-1}_{\Delta}(\cX)$.

  To verify this last claim, observe first that \Cref{eq:evers.mtr} can be recast as \Cref{eq:snssns} restricted to $\sigma^{-}_{ss}(\cX)$. That restriction is continuous precisely when it is so when further composed with the adjoint: $\Ad_S N\mapsto \Ad_S N^*$. This, though, is nothing but the map obtained by complex-conjugating all eigenvalues of an arbitrary semisimple operator, while leaving its eigenspaces unaffected; \emph{that} map is indeed \cite[Theorem 4.3, (ii) $\Leftrightarrow$ (iv)]{zbMATH06285212} continuous if and only if complex conjugation $\overline{(\bullet)}$ is of class $B^{n-1}_{\Delta}(\cX)$, finishing the proof. 
\end{proof}


\begin{theorem}\label{th:all}
  Let $n\in \bZ_{\ge 3}$ and $\cX\subseteq \bC$ a perfect space with all graphs $\Gamma^n(\mathbf{x})$, $\mathbf{x}\in \cC^n(\cX)$ containing $n$-cycles. Assume furthermore that
  \begin{itemize}[wide]
  \item $S_n$ acts transitively on the connected components of $\cC^n(\cX)$;
  \item and that action is not free (i.e. the isotropy groups are not trivial). 
  \end{itemize}
  The continuous CS preservers
  \begin{equation*}
    M_n(\bC)\supseteq\sigma^{-}(\cX)\xrightarrow{\phi}M_n(\bC)
  \end{equation*}
  are as follows. 
  
  \begin{enumerate}[(1),wide]
  \item If complex conjugation is not $C^{n-1}_{\Delta}(\cX)$, precisely the conjugations and transpose conjugations.

  \item If complex conjugation is $C^{n-1}_{\Delta}(\cX)$, precisely the conjugations, transpose conjugations and their compositions with the unique continuous extension of \Cref{eq:evers.mtr} to $\sigma^{-}(\cX)$.
  \end{enumerate}  
\end{theorem}
\begin{proof}
  One proceeds exactly as in proving \Cref{th:ss}, with one distinction: the argument now boils down to \Cref{eq:evers.mtr} extending to a continuous map $\sigma^{-}(\cX)\to M_n(\bC)$ if and only if $\overline{(\bullet)}\in C^{n-1}_{\Delta}(\cX)$, via an application of \cite[Proposition 4.5, (i) $\Leftrightarrow$ (iii)]{zbMATH06285212} (rather than the previously cited \cite[Theorem 4.3, (ii) $\Leftrightarrow$ (iv)]{zbMATH06285212}). 
\end{proof}


\addcontentsline{toc}{section}{References}

\def\polhk#1{\setbox0=\hbox{#1}{\ooalign{\hidewidth
  \lower1.5ex\hbox{`}\hidewidth\crcr\unhbox0}}}


\Addresses

\end{document}